\documentclass[11pt]{amsart}

\usepackage[T1]{fontenc}
\usepackage[utf8]{inputenc}
\usepackage{lmodern}
\usepackage[a4paper,margin=1.05in]{geometry}
\usepackage{amsmath,amssymb,amsthm,mathtools}
\usepackage{microtype}
\usepackage{enumitem}
\usepackage[colorlinks=true,linkcolor=blue,citecolor=blue,urlcolor=blue]{hyperref}
\hypersetup{
  pdftitle={Pseudo-Hyperjump Inversion Fails for Turing Degrees},
  pdfauthor={Patrizio Cintioli},
  pdfsubject={Computability theory; hyperjump inversion},
  pdfkeywords={Pseudo-hyperjump, hyperjump, Turing degrees,
pseudo-hyperjump inversion, hyperarithmetic reducibility, Kleene's O}
}

\newtheorem{theorem}{Theorem}[section]
\newtheorem{lemma}[theorem]{Lemma}
\newtheorem{corollary}[theorem]{Corollary}
\theoremstyle{definition}
\newtheorem{definition}[theorem]{Definition}
\newtheorem{remark}[theorem]{Remark}

\newcommand{\N}{\mathbb{N}}
\newcommand{\OO}{\mathcal{O}}
\newcommand{\HJ}{\operatorname{HJ}}
\newcommand{\degT}{\operatorname{deg}_{\mathrm T}}
\newcommand{\leT}{\leq_{\mathrm T}}
\newcommand{\nleT}{\nleq_{\mathrm T}}
\newcommand{\geT}{\geq_{\mathrm T}}
\newcommand{\ltT}{<_{\mathrm T}}
\newcommand{\eqT}{\equiv_{\mathrm T}}
\newcommand{\leH}{\leq_{\mathrm{HYP}}}
\newcommand{\ltH}{<_{\mathrm{HYP}}}
\newcommand{\nleH}{\nleq_{\mathrm{HYP}}}
\newcommand{\join}{\mathbin{\oplus}}
\newcommand{\CK}{\mathrm{CK}}

\title[Pseudo-Hyperjump Inversion]
      {Pseudo-Hyperjump Inversion Fails for Turing Degrees}
\author[P. Cintioli]{Patrizio Cintioli}
\address{Mathematics Division, School of Science and Technology, University of Camerino, Italy}
\email{patrizio.cintioli@unicam.it}


\begin{document}

\begin{abstract}
Jananthan and Simpson asked in Conjecture~4.5 whether every
pseudo-hyperjump
$\HJ_e(X)=X\join V_e^X$, with $V_e^X$ uniformly $\Pi^{1,X}_1$,
admits Turing-degree inversion above Kleene's $\OO$.
We settle their conjecture in the negative.  We construct one index
$e_*$ such that $X\ltT\HJ_{e_*}(X)$ for every real $X$, while the
range of $\HJ_{e_*}$ omits every Turing degree in the interval
\[
  \{\mathbf d:
    \degT(\OO)\leq \mathbf d<\degT(\OO^{\OO})\}.
\]
Thus Conjecture~4.5 fails even for an operator which strictly raises
the Turing degree of every input.  Each of Properties~4.6--4.8
implies Conjecture~4.5; consequently, the accompanying
class-characterisation problems collapse: no class of reals, even
without a definability assumption, satisfies any one of these
properties.
\end{abstract}

\maketitle

\section{Introduction}

We identify reals with subsets of $\N$ (equivalently, with elements of
$2^\N$), and use $\leT$, $\leH$, and $\leq_m$ for Turing,
hyperarithmetical, and many-one reducibility, respectively.
The corresponding strict and equivalence relations are denoted in the
usual way.  The Turing jump of $X$ is denoted by $X'$.
The ordinal $\omega_1^X$ is the least ordinal which is not the order
type of an $X$-recursive well-ordering, and
$\omega_1^{\CK}=\omega_1^{\emptyset}$ is the Church--Kleene ordinal,
the least noncomputable ordinal.
All analytical pointclasses are lightface, with any oracle parameter
indicated explicitly.  We write $\degT(X)$ for the Turing degree of
$X$, and $X\join Y$ for the usual effective join of two reals.

For a real $X$, let $\OO^X$ denote the relativized Kleene
$\OO$, the hyperjump of $X$.  It is $\Pi^{1,X}_1$-complete, uniformly
in $X$.  Equivalently, its Turing degree is that of the set of indices
of $X$-recursive well-orderings.  We put $\OO=\OO^{\emptyset}$.

Let $(V_e^X)_{e\in\N}$ be a standard effective
enumeration, uniform in $X$, of the $\Pi^{1,X}_1$ subsets of $\N$.
Jananthan and Simpson \cite{JananthanSimpson} define
\[
  \HJ_e(X)=X\join V_e^X
\]
and ask whether the following analogue of the pseudojump inversion
theorem holds \cite[Conjecture 4.5]{JananthanSimpson}:
\[
  \forall e\;\forall A\geT\OO\;\exists B\quad
  A\eqT\HJ_e(B)\eqT B\join\OO.
\]
Here and below $A\geT\OO$ abbreviates $\OO\leT A$.
We settle Conjecture~4.5 in the negative.  More precisely, we
construct one index $e_*$ whose pseudo-hyperjump strictly raises
every input degree but omits every degree in
\[
  \{\mathbf d:
    \degT(\OO)\leq\mathbf d<\degT(\OO^\OO)\}.
\]
Immediately after stating Conjecture~4.5, Jananthan and Simpson
observe that even an affirmative answer would leave open the
characterisation of the $\Sigma^1_1$ classes satisfying their
Properties~4.6--4.8.  Our negative answer has the opposite
consequence: as we observe in Section~3, each of Properties~4.6--4.8
implies Conjecture~4.5, and hence none of these properties is
satisfied by any class.

The construction uses the change in behaviour at the predicate
$\omega_1^X=\omega_1^{\CK}$.  The class of reals satisfying this
predicate is $\Sigma^1_1$; therefore its complement can be used as a
$\Pi^1_1$ switch inside a single pseudo-hyperjump.
References to the conjecture and numbered properties of
\cite{JananthanSimpson} use arXiv:2101.08818v2 (19 July 2024).
The inversion question already appears in Simpson's 2019 NUS lecture,
reporting joint work with Jananthan \cite[slide 9]{SimpsonNUS}.

We use the following standard facts about the hyperjump:
\begin{enumerate}[label=(\roman*)]
  \item $X\ltH\OO^X$; this follows from the completeness of
        $\OO^X$ and \cite[Theorem II.7.6(i), p.~51]{Sacks}.
  \item If $Y\leH X$, then $\OO^Y\leT\OO^X$;
        in fact, $\OO^Y\leq_m\OO^X$
        \cite[Proposition II.7.1, p.~49]{Sacks}.
  \item $\omega_1^X>\omega_1^{\CK}$ if and only if
        $\OO\leH X$
        \cite[Corollary II.7.7, p.~51]{Sacks}.
\end{enumerate}
The uniform completeness of the usual hyperjump also gives
$X\join X'\leT\OO^X$ uniformly in $X$
\cite[Subsection II.5.4, p.~43]{Sacks}.

\section{A pseudo-hyperjump with a gap}

The construction uses two columns with different roles.  The even
column uniformly codes $X'$, ensuring strict Turing-degree increase,
while the odd column codes $\OO^X$ precisely when
$\omega_1^X>\omega_1^{\CK}$.  We first verify that the switch between
these two cases has the required effective descriptive complexity.

\begin{lemma}\label{lem:complexity}
The class
\[
  \mathcal L=\{X:\omega_1^X=\omega_1^{\CK}\}
\]
is $\Sigma^1_1$
\cite[Corollary III.1.5, p.~54]{Sacks}.
Consequently, the predicate
\[
  \omega_1^X>\omega_1^{\CK}
\]
is a $\Pi^1_1$ predicate of $X$.
\end{lemma}

\begin{proof}
Let $(\varphi_e^X)_{e\in\N}$ be a uniform effective enumeration
of the partial $X$-recursive functions from $\N^2$ to $\{0,1\}$,
viewed as putative characteristic functions of strict linear
orderings on $\N$. Write $R_e^X$ for the corresponding candidate
presentation, and let $R_i=R_i^\emptyset$.

Since $\omega_1^{\CK}\leq\omega_1^X$ for every $X$, we have
$\omega_1^X=\omega_1^{\CK}$ if and only if, for every $e$, either
$R_e^X$ does not present a well-ordering or there are $i$ and a real
$f$ such that $R_i$ is a total linear ordering and $f$ is an
isomorphism from $R_i$ onto $R_e^X$.
Indeed, in the equality case every $X$-recursive well-ordering has
recursive order type.  If $\omega_1^X>\omega_1^{\CK}$, an
$X$-recursive presentation of $\omega_1^{\CK}$ has no recursive copy.

Failure to present a well-ordering is a $\Sigma^1_1(X)$ condition:
failure to be a total linear ordering is arithmetical in $X$, while
ill-foundedness is witnessed by an infinite descending sequence.
The existence of $i$ and of the isomorphism $f$ is also
$\Sigma^1_1(X)$.  Finally, the universal number quantifier over $e$
can be absorbed into a single existentially quantified real coding
a sequence of tagged witnesses.  The resulting definition is
uniform and effective, so the class is $\Sigma^1_1$.
\end{proof}

\begin{definition}\label{def:Vstar}
Uniformly in $X$, define $V_*^X\subseteq\N$ by
\begin{align*}
  2n\in V_*^X
  &\iff n\in X',\\
  2n+1\in V_*^X
  &\iff \omega_1^X>\omega_1^{\CK}\ \&\ n\in\OO^X.
\end{align*}
\end{definition}

The first clause is arithmetical in $X$, and hence can be written as a
$\Pi^{1,X}_1$ predicate. By Lemma~\ref{lem:complexity}, the second clause
is the conjunction of two $\Pi^{1,X}_1$ predicates.
Hence, using the arithmetical parity distinction, $V_*^X$ is
uniformly $\Pi^{1,X}_1$.
Fix an index $e_*$ such that $V_{e_*}^X=V_*^X$ for every real $X$.

\begin{theorem}[Gap theorem]\label{thm:gap}
For every real $X$,
\[
  X\ltT\HJ_{e_*}(X).
\]
Moreover, the following dichotomy holds:
\[
\begin{array}{ll}
\omega_1^X=\omega_1^{\CK}
&\Longrightarrow
\HJ_{e_*}(X)\eqT X'
\text{ and }
\OO\nleT\HJ_{e_*}(X),\\[1mm]
\omega_1^X>\omega_1^{\CK}
&\Longrightarrow
\HJ_{e_*}(X)\eqT\OO^X
\text{ and }
\OO^\OO\leT\HJ_{e_*}(X).
\end{array}
\]

Consequently,
\[
  \{\degT(\HJ_{e_*}(X)):X\in 2^\N\}
  \cap
  \{\mathbf d:
     \degT(\OO)\leq\mathbf d<\degT(\OO^\OO)\}
  =\varnothing.
\]
\end{theorem}

\begin{proof}
The even column gives $X'\leT\HJ_{e_*}(X)$ for every $X$.
Thus $X\ltT X'\leT\HJ_{e_*}(X)$, proving strictness.

Suppose first that $\omega_1^X=\omega_1^{\CK}$.  The odd column in
Definition~\ref{def:Vstar} is empty, so
\[
  \HJ_{e_*}(X)=X\join V_*^X\eqT X'.
\]
Spector's criterion gives $\OO\nleH X$.  Since $X'\leH X$,
we have $\OO\nleT X'$; otherwise $\OO\leT X'\leH X$,
a contradiction.  This proves the first alternative.

Now suppose that $\omega_1^X>\omega_1^{\CK}$.  The odd column of
$V_*^X$ is a copy of $\OO^X$, which computes both $X$ and $X'$;
thus $\HJ_{e_*}(X)\eqT\OO^X$.
Spector's criterion gives $\OO\leH X$.  Applying the implication
$Y\leH X \Longrightarrow \OO^Y\leT\OO^X$ with $Y=\OO$, we obtain
\[
  \OO^{\OO}\leT\OO^X\eqT\HJ_{e_*}(X).
\]
The two alternatives show that the stated degree interval is omitted.
\end{proof}

\begin{corollary}\label{cor:conjecture-false}
For the index $e_*$, if
\[
  \OO\leT A\ltT\OO^\OO,
\]
then there is no real $B$ such that
\[
  A\eqT\HJ_{e_*}(B).
\]
Consequently, Conjecture~4.5 of
\cite{JananthanSimpson} is false.  In fact, it remains false if its
hypothesis $\OO\leT A$ is strengthened to $\OO\ltT A$.
\end{corollary}

\begin{proof}
The first assertion follows immediately from
Theorem~\ref{thm:gap}.  Moreover,
\[
  \OO\ltT\OO'\ltT\OO^\OO.
\]
Here $\OO\ltT\OO'$ is the strictness of the Turing jump, while
$\OO'\leT\OO^\OO$ follows from the uniform completeness of the
hyperjump.  If $\OO^\OO\leT\OO'$, then, since
$\OO'\leH\OO$, we would have $\OO^\OO\leH\OO$, contradicting
$\OO\ltH\OO^\OO$.

Thus, taking $A=\OO'$ gives a counterexample with
$\OO\ltT A$, and in particular refutes Conjecture~4.5.
\end{proof}

\begin{remark}\label{rem:strength}
The use of the even $X'$-column is not needed merely to refute
Conjecture~4.5; it ensures the stronger conclusion
$X\ltT\HJ_{e_*}(X)$ for every $X$.
Thus the conjecture remains false even when restricted to
pseudo-hyperjump operators that strictly increase the Turing
degree of every input.
\end{remark}

\section{The class properties for pseudo-hyperjumps}

We now turn to the consequences of the failure of Conjecture~4.5 for
the associated class-characterisation questions.  We first recall
Properties~4.6--4.8 and then examine their logical relation to
Conjecture~4.5.

Jananthan and Simpson ask for a characterisation of the
$\Sigma^1_1$ classes $K\subseteq 2^\N$ satisfying the following
three properties \cite[Properties 4.6--4.8]{JananthanSimpson}.
We state them explicitly, using $\emptyset$ in place of their $0$.
For columns we use their convention
\[
  (Z)_k=\{n:2^k3^n\in Z\}.
\]

\medskip
\noindent\textbf{Property 4.6.}
Suppose $e\in\N$ and $A$ is a real such that $\OO\leT A$.
Then there exists $B\in K$ such that
\[
  A\eqT\HJ_e(B)\eqT B\join\OO.
\]

\medskip
\noindent\textbf{Property 4.7.}
Suppose $e\in\N$ and $Z$ and $A$ are reals such that
\[
  Z\join\OO\leT A
  \qquad\text{and}\qquad
  \emptyset\ltH(Z)_k
  \quad\text{for every }k\in\N.
\]
Then there exists $B\in K$ such that
\[
  A\eqT\HJ_e(B)\eqT B\join\OO
\]
and
\[
  (Z)_k\nleH B
  \quad\text{for every }k\in\N.
\]

\medskip
\noindent\textbf{Property 4.8.}
Suppose $e\in\N$ and $Z$ and $A$ are reals such that
\[
  Z\join\OO\leT A
  \qquad\text{and}\qquad
  \emptyset\ltH Z.
\]
Then there exists $B\in K$ such that
\[
  A\eqT\HJ_e(B)\eqT B\join Z\eqT B\join\OO.
\]
The logical relation between these properties and Conjecture~4.5 is
worth making explicit.  Each of Properties~4.6--4.8 implies
Conjecture~4.5.

For Property~4.6 this is immediate by forgetting the requirement
that the witness $B$ belong to $K$.  For Property~4.7, given
$e\in\N$ and $A\geT\OO$, let
\[
  Z=\{2^k3^n:k\in\N,\ n\in\OO\}.
\]
Then $Z\eqT\OO$, every column $(Z)_k$ equals $\OO$, and hence
\[
  Z\join\OO\leT A
  \qquad\text{and}\qquad
  \emptyset\ltH(Z)_k
  \quad\text{for every }k.
\]
Thus Property~4.7 yields a $B\in K$ such that
\[
  A\eqT\HJ_e(B)\eqT B\join\OO.
\]
For Property~4.8, given $e\in\N$ and $A\geT\OO$, take $Z=\OO$.
Again its hypotheses are satisfied, and its conclusion contains
\[
  A\eqT\HJ_e(B)\eqT B\join\OO.
\]

Thus the observation of Jananthan and Simpson that, even if
Conjecture~4.5 were true, the class-characterisation questions would
remain open reflects an asymmetry between affirmative and negative answers:
an affirmative answer to Conjecture~4.5 would not by itself characterise the
$\Sigma^1_1$ classes satisfying Properties~4.6--4.8, whereas a
negative answer forces each of the three collections of classes to
be empty.

\begin{corollary}\label{cor:no-classes}
No class $K\subseteq 2^\N$, with or without a definability
assumption, satisfies any one of Properties~4.6, 4.7, or 4.8 of
\cite{JananthanSimpson}.
\end{corollary}

\begin{proof}
By the preceding observation, the existence of such a class for any
one of Properties~4.6--4.8 would imply Conjecture~4.5, contradicting
Corollary~\ref{cor:conjecture-false}.
\end{proof}

\section{Conclusion}

We have shown that pseudo-hyperjump inversion fails at the level of
Turing degrees.  The failure is witnessed by a single pseudo-hyperjump
$\HJ_{e_*}$ which strictly increases the Turing degree of every input,
while its range omits every degree in the interval
\[
  \{\mathbf d:
    \degT(\OO)\leq\mathbf d<\degT(\OO^\OO)\}.
\]
Thus the obstruction is stronger than the failure of inversion at the
single target $\OO$: since
\[
  \OO\ltT\OO'\ltT\OO^\OO,
\]
Conjecture~4.5 remains false even when its target is required to lie
strictly above $\OO$.

The construction is obtained by using the predicate
$\omega_1^X=\omega_1^{\CK}$ as a uniform effective switch.  On one
side of the dichotomy the resulting pseudo-hyperjump has the Turing
degree of $X'$, while on the other it has the degree of $\OO^X$.
This produces the omitted interval while preserving strict
Turing-degree increase.  Finally, because each of
Properties~4.6--4.8 implies Conjecture~4.5, the negative solution also
settles the corresponding class-characterisation questions: no class
of reals, with or without a definability assumption, satisfies any one
of those properties.

\section*{Acknowledgements}
The main results were discovered by ChatGPT (OpenAI).  The author takes
full responsibility for the mathematical content of the paper.

\end{document}